\documentclass[12pt,reqno]{article}

\usepackage[usenames]{color}
\usepackage[colorlinks=true,
linkcolor=webgreen, filecolor=webbrown,
citecolor=webgreen]{hyperref}

\definecolor{webgreen}{rgb}{0,.5,0}
\definecolor{webbrown}{rgb}{.6,0,0}

\usepackage{amssymb}
\usepackage{graphicx}
\usepackage{amscd}
\usepackage{lscape}
\usepackage{tikz}
\usepackage{tikz-cd}
\usepackage{pgfplots}

\usetikzlibrary{matrix}
\usetikzlibrary{fit,shapes}
\usetikzlibrary{positioning, calc}
\tikzset{circle node/.style = {circle,inner sep=1pt,draw, fill=white},
        X node/.style = {fill=white, inner sep=1pt},
        dot node/.style = {circle, draw, inner sep=5pt}
        }
\usepackage{tkz-fct}

\usepackage{amsthm}
\newtheorem{theorem}{Theorem}
\newtheorem{lemma}[theorem]{Lemma}
\newtheorem{proposition}[theorem]{Proposition}
\newtheorem{corollary}[theorem]{Corollary}

\theoremstyle{definition}

\newtheorem{example}[theorem]{Example}

\usepackage{float}

\usepackage{graphics,amsmath}
\usepackage{amsfonts}
\usepackage{latexsym}
\usepackage{epsf}

\newcommand{\seqnum}[1]{\href{http://oeis.org/#1}{\underline{#1}}}

\begin{document}

\begin{center}
\vskip 1cm{\LARGE\bf On the Gap-sum and Gap-product Sequences of Integer Sequences} \vskip 1cm \large
Paul Barry\\
School of Science\\
Waterford Institute of Technology\\
Ireland\\
\href{mailto:pbarry@wit.ie}{\tt pbarry@wit.ie}
\end{center}
\vskip .2 in

\begin{abstract} In this note, we explore two families of sequences associated to a suitable integer sequence: the gap-sum sequence and the gap-product sequence. These are the sums and the products of consecutive numbers not in the original sequence. We give closed forms for both, in terms of the original sequence, and in the case of Horadam sequences, we find the generating function of the gap-sum sequence. For some elementary sequences, we indicate that the gap-product sequences are given by the Fuss-Catalan-Raney numbers.
\end{abstract}

\section{Introduction}
Let $A$ be the integer sequence $(a_n)_{n \ge 0}$. The consecutive numbers between, but not including, $a_n$ to $a_{n+1}$ will be called the $n$th $A$-gap, or just the $n$th gap, when the sequence in question is clear. Note that for sequences with $a_{n+1} > a_n+1$, each gap is an arithmetic sequence with initial term $a_n+1$, with $a_{n+1}-a_n-1$ terms, and increment $+1$. If $a_{n+1}=a_n+1$, then we say that the gap is $0$.
\begin{example} We consider the sequence $a_n=J_{n+2}$ which begins
$$1, 3, 5, 11, 21, 43, 85, 171,\ldots.$$ Here, $J_n$ is the $n$-th Jacobsthal number, and hence
$$a_n=4\frac{2^n}{3}-\frac{(-1)^n}{3}.$$
The first gaps of this sequence are then
$$1,\underbrace{2},3,\underbrace{4},5,\underbrace{6,7,8,9,10},11,\underbrace{12,13,\cdots,20},21,\ldots.$$
\end{example}
We define \emph{the gap-sum sequence of $a_n$} to be the sequence whose $n$-th term is the sum of the elements of the $n$th-gap. Thus for the sequence above, the gap-sum sequence begins
$$2, 4, 40,144,\ldots.$$ This sequence is often called the sequence of sums of consecutive non-$A$ numbers.
Similarly, we define the \emph{gap-product sequence of $a_n$} to be the sequence whose $n$-th term is the product of the elements of the $n$-th gap. For the sequence of the above example, we have that the gap-product sequence begins
$$2,4, 30240,60949324800,\ldots.$$
A first result is the following.
\begin{proposition} Let $S_n$ be the gap-sum sequence of $a_n$. Then we have
$$S_n=\sum_{j=1}^{a_{n+1}-a_n-1} (a_n+j).$$
Let $P_n$ be the gap-product sequence of $a_n$. Then we have
$$P_n=\prod_{j=1}^{a_{n+1}-a_n-1} (a_n+j).$$
\end{proposition}
\begin{proof}
This is a reformulation of the definitions of the gap-sum and the gap-product sequences.
\end{proof}
\begin{example}
We consider the Fibonacci numbers \seqnum{A000045}
$$F_n=\frac{1}{\sqrt{5}}\left(\frac{1}{2}+\frac{\sqrt{5}}{2}\right)^n-\frac{1}{\sqrt{5}}\left(\frac{1}{2}-\frac{\sqrt{5}}{2}\right)^n.$$
We find that the formula for the  gap-sum sequence $\sum_{j=1}^{F_{n+1}-F_n-1} F_n+j$ above in this case gives us
$$0, 0, 0, 0, 4, 13, 42, 119, 330, 890, 2376,\ldots.$$
This is \seqnum{A109454} in the On-Line Encyclopedia of Integer Sequences (OEIS) \cite{SL1, SL2}.
\end{example}
\begin{example}
We let $A$ denote the sequence \seqnum{A000040} of prime numbers
$$2, 3, 5, 7, 11, 13, 17, 19, 23, 29, 31, 37, 41, 43, 47,\ldots.$$ The prime gap-sum sequence then begins
$$0, 4, 6, 27, 12, 45, 18, 63, 130, 30, 170,\ldots.$$ This is the sequence \seqnum{A054265} in the OEIS, where it is described as the sequence of sums of composite numbers between successive primes.
\end{example}
Using the formula for the sum of an arithmetic series, we obtain the following result for the gap-sum sequence.
\begin{proposition}
Let $a_n$ be a sequence such that $a_{n+1} > a_n$. Then
$$S_n = \frac{a_{n+1}-a_{n}-1}{2} \left(a_n + a_{n+1}\right).$$
\end{proposition}
\begin{proof} The $n$-th gap is an arithmetic sequence beginning with $a_n+1$ and having step $d=1$, with $a_{n+1}-a_n-1$ terms.
Thus we have
$$S_n = \frac{a_{n+1}-a_n-1}{2}\left(2(a_n+1)+(a_{n+1}-a_n-1)-1\right)=\frac{a_{n+1}-a_n-1}{2}\left(a_n+a_{n+1}\right).$$
\end{proof}
For instance, if $p_n$ denotes the $n$-th prime, then the prime gap-sum sequence has general term
$$S_n = \frac{p_{n+1}-p_n-1}{2}\left(p_n+p_{n+1}\right).$$
We note that the gap-sum sequence can be the zero sequence. For example, if $a_n=n$, then the gap-sum sequence is the zero sequence
$$0,0,0,0,\ldots.$$
\begin{corollary} Let $a_n=rn$, $r \in \mathbb{N}$. Then the gap-sum sequence of $a_n$ is given by
$$S_n=\frac{1}{2}(2n+1)r(r-1)=(2n+1)\binom{r}{2}.$$
\end{corollary}
\begin{corollary} Let $a_n=k^n$, then the gap-sum sequence for $a_n$ is given by
$$S_n = \frac{1}{2}\left((k^2-1)k^{2n}-(k+1)k^n\right),$$ with generating function
$$S(x)=\frac{1}{2}\frac{(k+1)(kx+k-2)}{(1-kx)(1-k^2x)}.$$
\end{corollary}
For $k=2$, this is the gap-sum sequence for $2^n$. This is \seqnum{A103897}, which begins
$$0,3, 18, 84, 360, 1488, 6048, 24384,\ldots.$$
This sequence has generating function
$$\frac{3x}{(1-2x)(1-4x)}.$$
We note that if $a_n=2^n-1$, then $S_n$ is the sequence that begins
$$0, 2, 15, 77, 345, 1457, 5985,\ldots.$$
This sequence has the generating function
$$\frac{x(2+x)}{(1-x)(1-2x)(1-4x)}.$$
\begin{example} We consider the case of $a_n=2n^2$. We find that
$$S_n=8n^3+10n^2+6n+1,$$ with generating function
$$\frac{1+21x+33x^2+3x^3}{(1-x)^4}.$$
This sequence begins
$$1, 25, 117, 325, 697, 1281, 2125, 3277, 4785,\ldots.$$
\end{example}
The gap-sum sequence for $a_n=n^2$ is \seqnum{A048395}. For this sequence we have 
$$S_n=2n^3 + 2n^2 + n,$$ with generating function 
$$\frac{x(1+x)(5+x)}{(1-x)^4}.$$ 
We give below a short table of figurate numbers and their gap-sum sequences.
\begin{center}\begin{tabular}{|c|c|c|c|}\hline
G.f. of $a_n$ & $a_n$ & $S_n$ & G.f. of $S_n$ \\ \hline
$\frac{x(1+x)}{(1-x)^3}$ & $n^2$ & $2n^3 + 2n^2 + n$ & $\frac{x(1+x)(5+x)}{(1-x)^4}$ \\ \hline
$\frac{x}{(1-x)^3}$ & $\frac{n(n+1)}{2}$ & $\frac{n(n+1)^2}{2}$ & $\frac{x(2+x)}{(1-x)^4}$ \\ \hline
$\frac{2x}{(1-x)^3}$ & $n(n+1)$ & $(2n+1)(n+1)^2$ & $\frac{1+8x+3x^2}{(1-x)^4}$ \\ \hline
$\frac{x(2+x)}{(1-x)^3}$ & $\frac{n(3n+1)}{2}$ & $\frac{(3n+1)(3n^2+4n+2)}{2}$ & $\frac{1+14x+11x^2+x^3}{(1-x)^4}$ \\ \hline
$\frac{x}{(1-x)^4}$ & $\frac{n}{6}(n+1)(n+2)$ & $\frac{n(n+1)(n+2)(n+3)(2n+3)}{4!}$ & $\frac{5x(1+x)}{(1-x)^6}$ \\ \hline
$\frac{x}{(1-x)^5}$ & $\binom{n+3}{4}$ & $\frac{n(n+1)(n+2)^2(n+3)(n^2+6n+11)}{144}$ & $\frac{x(9+18x+7x^2+x^3)}{(1-x)^8}$ \\ \hline
$\frac{x(1+2x)}{(1-x)^3}$ & $\frac{n(3n-1)}{2}$ & $\frac{3n(3n^2+2n+1)}{3}$ & $\frac{3x(x^2+5x+3)}{(1-x)^4}$ \\ \hline
\end{tabular}\end{center}
The gap-sum sequence of the triangular numbers $\frac{n(n+1)}{2}=\binom{n+1}{2}$ is \seqnum{A006002}. 

The gap-sum sequence of the tetrahedral numbers $\frac{n}{6}(n+1)(n+2)=\binom{n+2}{3}$ can be expressed as 
$$5 \binom{n+3}{4}+10 \binom{n+3}{5}.$$ 
This is $5$ times \seqnum{A005585}, the $5$-dimensional pyramidal numbers $\frac{n(n+1)(n+2)(n+3)(2n+3)}{5!}$. 

The gap-sum sequence of $\binom{n+3}{4}$ can be expressed as 
$$9 \binom{n+4}{5}+36 \binom{n+4}{6}+34 \binom{n+4}{7}+\binom{n+3}{7}.$$
\section{The gap-product sequence}
\begin{proposition} Let $a_n$ be a sequence with $a_{n+1} \ge a_n$. Then the gap-product sequence $P_n$ is given by $$P_n=\frac{(a_{n+1}-1)!}{a_n!}.$$
\end{proposition}
\begin{proof}
The product of the terms of an arithmetic sequence with initial term $t_1$, with $r$ terms and increment $d$, is given by
$$P=\frac{\Gamma\left(\frac{t_1}{d}+r\right)}{\Gamma(t_1)}.$$
For $P_n$, we have $t_1=a_n+1$, $r=a_{n+1}-a_n-1$, and $d=1$.
The result follows from this.
\end{proof}
\begin{example} We consider the Fibonacci numbers (whose terms satisfy $a_{n+1}>a_n$ after $n=2$). We obtain that the gap-product sequence
$$P_n = \frac{(F_{n+1}-1)!}{F_n!}$$ begins
$$1, 1, 1, 1, 4, 42, 11880, 390700800, 169958063987712000,\ldots.$$
Since $F{n+1}=F_n+F_{n-1}$, we have that
$$P_n=\frac{\prod_{j=1}^{F_{n-1}-1}(F_{n+1}-j) F_n!}{F_n!}=\prod_{j=1}^{F_{n-1}-1}(F_{n+1}-j).$$
\end{example}
\begin{example} For the sequence $a_n=4\frac{2^n}{3}-\frac{(-1)^n}{3}$, we have
$$P_n=\frac{(8\frac{2^n}{3}+\frac{(-1)^n}{3}-1)!}{(4\frac{2^n}{3}-\frac{(-1)^n}{3})!}=\frac{(J_{n+3}-1)!}{J_{n+2}!},$$ where $J_n=\frac{2^n}{3}-\frac{(-1)^n}{3}$ is the $n$-th Jacobsthal number (\seqnum{A001045}).
\end{example}
\section{The gap-sum sequences of Horadam sequences}
A Horadam sequence is a second-order recurrence sequence $a_n=a_n(\alpha, \beta, r, s)$  that is  defined by
$$ a_n=r a_{n-1}+ s a_{n-2},$$
with $a_0=\alpha$, $a_1=\beta$. For instance, $F_n=F_n(0,1,1,1)$ and $J_n=J_n(0,1,1,2)$. The generating function of $a_na_n(\alpha, \beta, r, s)$ is given by
$$\frac{\alpha-(\alpha r-\beta)x}{1-rx-sx^2}.$$
In this section we wish to find the generating function of the gap-sum sequence for a general Horadam sequence. We know that the gap-sum sequence for a general sequence $a_n$ is given by
$$S_n=\frac{1}{2}\left(a_n+a_{n+1}\right)\left(a_{n+1}-a_n-1\right)=\frac{1}{2}\left(a_{n+1}^2-a_n^2-a_n-a_{n+1}\right).$$ Thus the generating function of $S_n$ is determined by those of $a_n$, $a_{n+1}$, $a_n^2$, and $a_{n+1}^2$. In the case of Horadam sequences, the sequence $a_{n+1}$ is again a Horadam sequence, with generating function
 $$\frac{\beta+\alpha s x}{1-rx-sx^2}, $$ with initial terms $\beta, r \alpha+ s \beta$. Thus
 $$a_{n+1}(\alpha, \beta, r, s)=a_n(\beta, r \alpha+ s \beta, r, s).$$
 Thus we are left with the problem of finding the generating function of the square of a Horadam sequence. This question has been addressed in \cite{Mansour}.
\begin{lemma} The generating function of $a_n^2(\alpha, \beta, r, s)$ is given by
$$g_2(\alpha, \beta, r, s)=\frac{\alpha^2-(\alpha^2(r^2+s)-\beta^2)x-s(\alpha^2 r^2- 2 \alpha \beta r+beta^2)x^2}{1-(r^2+s)x-s(r^2+s)x^2+s^3x^3},$$ and the generating function of $a_{n+1}^2(\alpha, \beta, r, s)$ is given by
$$\frac{\beta^2+s(\alpha^2 s+2 \alpha \beta r-\beta^2)x-\alpha^2 s^3 x^2}{1-(r^2+s)x-s(r^2+s)x^2+s^3 x^3}.$$
\end{lemma}
\begin{example} The generating function of the square of $J_{n+2}$ is given by
$$g_2(1,3,1,2)=\frac{1+6x-8x^2}{1-3x-6x^2+8x^3}.$$
\end{example}
The lemma then leads to the following result.
\begin{proposition} The generating sequence of the gap-sum sequence of the Horadam sequence $a_n=a_n(\alpha, \beta, r, s)$ is given as follows. Let
$$W=\frac{\beta^2+s(\alpha^2 s+2 \alpha \beta r-\beta^2)x-\alpha^2 s^3 x^2}{1-(r^2+s)x-s(r^2+s)x^2+s^3 x^3}.$$
Let
$$X=\frac{\alpha^2-(\alpha^2(r^2+s)-\beta^2)x-s(\alpha^2 r^2- 2 \alpha \beta r+beta^2)x^2}{1-(r^2+s)x-s(r^2+s)x^2+s^3x^3}.$$
Let
$$Y= \frac{\beta+ \alpha s x}{1-rx-sx^2},$$ and
let
$$Z=\frac{\alpha-(\alpha r-\beta)x}{1-rx-sx^2}.$$

Then the generating function of the gap-sum sequence of the Horadam sequence $a_n=a_n(\alpha, \beta, r, s)$ \seqnum{A002605} which begins
$$1, 2, 6, 16, 44, 120, 328,\ldots,$$ is given by
$$W-X-Y-Z.$$
\end{proposition}
\begin{example} For the Horadam sequence $a_n(1,2,2,2)$ we find that the gap-sum sequence has generating function $$\frac{3x(4+x-2x^2)}{(1-2x-2x^2)(1-6x-12x^2+8x^3)}.$$
This expands to give the sequence $S_n$ that begins
$$0, 12, 99, 810, 6150, 46368, 347004,\ldots.$$
\end{example}
\begin{example} The table below summarizes the cases of the Fibonacci numbers $F_{n+1}=\sum_{k=0}^{\lfloor \frac{n}{2} \rfloor} \binom{n-k}{k}$, the Jacobsthal numbers $J_{n+1}=\sum_{k=0}^{\lfloor \frac{n}{2} \rfloor} \binom{n-k}{k}2^k$, and the Pell numbers $P_{n+1}=\sum_{k=0}^{\lfloor \frac{n}{2} \rfloor} \binom{n-k}{k}2^{n-2k}$ \seqnum{A000129}. Note that the sequence from the previous example has $a_n(1,2,2,2)=\sum_{k=0}^{\lfloor \frac{n}{2} \rfloor} \binom{n-k}{k}2^{n-k}$.

\begin{center}\begin{tabular}{|c|c|c|c|}\hline
Sequence & G.f. & G.f. of gap-sum & First terms of gap-sum\\ \hline
$F_{n+1}$ & $\frac{1}{1-x-x^2}$ & $-\frac{1-3x-x^2+x^3}{(1-x-x^2)(1-2x-2x^2+x^3)}$ & $-1, 0, 0, 4, 13, 42, 119, 330,\ldots.$ \\ \hline
$J_{n+1}$ & $\frac{1}{1-x-2x^2}$ & $-\frac{1-6x}{(1-2x)(1-2x-8x^2)}$ & $-1, 2, 4, 40, 144, 672, 2624, 10880$ \\ \hline
$P_{n+1}$ & $\frac{1}{1-2x-x^2}$ & $\frac{x(7+2x-x^2)}{(1-2x-x^2)(1-5x-5x^2+x^3)}$ & $0, 7, 51, 328, 1980, 11711, 68663, 401184,\ldots$ \\ \hline
\end{tabular}\end{center}
We note that the $-1$'s are artifacts arising from consecutive elements of the sequences in question being equal.
\end{example}

\section{A special gap-product sequence related to the Fuss-Catalan numbers}
We have seen that the gap-product sequence of the sequence $a_n$ is given by
$$P_n=\frac{(a_{n+1}-1)!}{a_n!}.$$
We take the special case of the one-parameter sequence
$$a_n=kn+1.$$
We then have that $P_n$ (which is again a one-parameter sequence) is given by
$$P_n = \frac{(k(n+1)+1-1)!}{(kn+1)!}=\frac{(k(n+1))!}{(kn+1)!}.$$
\begin{proposition}
The gap-product sequence of the sequence $a_n=kn+1$ is given by
$$P_n=k! FC(k,n)$$ where $F(n,k)=\frac{1}{kn+1}\binom{(k+1)n}{n}$ is the $n,k$ Fuss Catalan number.
\end{proposition}
\begin{proof} This follows since
$$P_n = \frac{(k(n+1)+1-1)!}{(kn+1)!}=\frac{(k(n+1))!}{(kn+1)!}=\frac{k!}{k+1}\binom{(n+1)k}{k}.$$
\end{proof}
In the following table, we show the gap-product sequences corresponding to the sequences $kn+1$ on the left.
\begin{center}\begin{tabular} {|c|ccccccc|}\hline
$1$ & $1$ & $1$ & $1$ & $1$ & $1$ & $1$ & $\dots$ \\ \hline
$n+1$ & $1$ & $1$ & $1$ & $1$ & $1$ & $1$ & $\dots$ \\ \hline
$2n+1$ & $2$ & $4$ & $6$ & $8$ & $10$ & $12$ & $\dots$ \\ \hline
$3n+1$  & $6$ & $30$ & $72$ & $132$ & $210$ & $306$ & $\dots$ \\ \hline
$4n+1$ & $24$ & $336$ & $1320$ & $6840$ & $12144$ & $19656$ & $\dots$ \\ \hline
$5n+1$ & $120$ & $5040$ & $32760$ & $116280$ & $303600$ & $657720$ & $\dots$ \\ \hline
\end{tabular}\end{center}
The following is a table of the corresponding Fuss-Catalan numbers.
\begin{center}\begin{tabular} {|c|ccccccc|}\hline
& $\binom{n}{n}$ & $\frac{1}{n+1}\binom{2n}{n}$ & $\frac{1}{2n+1}\binom{3n}{n}$ & $\frac{1}{3n+1}\binom{4n}{n}$ &$\frac{1}{4n+1}\binom{5n}{n}$ & $\frac{1}{5n+1}\binom{6n}{n}$ & $\ldots$  \\ \hline
$1$ & $1$ & $1$ & $1$ & $1$ & $1$ & $1$ & $\dots$ \\ \hline
$n+1$ & $1$ & $1$ & $1$ & $1$ & $1$ & $1$ & $\dots$ \\ \hline
$2n+1$ & $1$ & $2$ & $3$ & $4$ & $5$ & $6$ & $\dots$ \\ \hline
$3n+1$  & $1$ & $5$ & $12$ & $22$ & $35$ & $51$ & $\dots$ \\ \hline
$4n+1$ & $1$ & $14$ & $55$ & $140$ & $285$ & $506$ & $\dots$ \\ \hline
$5n+1$ & $1$ & $42$ & $273$ & $969$ & $2530$ & $5481$ & $\dots$ \\ \hline
\end{tabular}\end{center}
Reading left to right, these sequences are \seqnum{A000012}, \seqnum{A000108}, \seqnum{A001764}, \seqnum{A002293}, \seqnum{A002294} and \seqnum{A002295}.
\begin{example} We take the one-parameter sequence $a_n=kn+2$. We find that
$$P_n = \frac{k! \binom{(n+1)k+1}{k}}{kn+2}.$$
For $k \ge 0$, this gives us the array that begins
\begin{center}\begin{tabular} {|c|ccccccc|}\hline
$2$ & $1/2$ & $1/2$ & $1/2$ & $1/2$ & $1/2$ & $1/2$ & $\dots$ \\ \hline
$n+2$ & $1$ & $1$ & $1$ & $1$ & $1$ & $1$ & $\dots$ \\ \hline
$2n+2$ & $3$ & $5$ & $7$ & $9$ & $11$ & $13$ & $\dots$ \\ \hline
$3n+2$  & $12$ & $42$ & $90$ & $156$ & $240$ & $342$ & $\dots$ \\ \hline
$4n+2$ & $60$ & $504$ & $1716$ & $4080$ & $7980$ & $13800$ & $\dots$ \\ \hline
$5n+1$ & $360$ & $7920$ & $43680$ & $143640$ & $358800$ & $755160$ & $\dots$ \\ \hline
\end{tabular}\end{center}
The first twos sequences here are \seqnum{A001710} and \seqnum{A102693}.
Dividing $P_n$ by $\frac{k!}{2}$ now gives us the array with general element $\frac{2 \binom{(n+1)k+1}{k}}{kn+2}$ that begins
\begin{center}\begin{tabular} {|c|ccccccc|}\hline
& $\frac{2\binom{n+1}{n}}{2}$ & $\frac{2\binom{2n+1}{n}}{n+2}$ & $\frac{2 \binom{3n+1}{2}}{2n+2}$ & $\frac{2\binom{4n+1}{n}}{3n+2}$ &$\frac{2\binom{5n+1}{n}}{4n+2}$ & $\frac{2\binom{6n+1}{n}}{5n+2}$ & $\ldots$  \\ \hline
$\frac{2\binom{1}{0}}{2}$ & $1$ & $1$ & $1$ & $1$ & $1$ & $1$ & $\dots$ \\ \hline
$\frac{2\binom{n+2}{1}}{n+2}$ & $2$ & $2$ & $2$ & $2$ & $2$ & $2$ & $\dots$ \\ \hline
$\frac{2 \binom{2n+3}{2}}{2n+2}$ & $3$ & $5$ & $7$ & $9$ & $11$ & $13$ & $\dots$ \\ \hline
$\frac{2\binom{3n+4}{3}}{3n+2}$  & $4$ & $14$ & $30$ & $52$ & $80$ & $114$ & $\dots$ \\ \hline
$\frac{2\binom{4n+5}{4}}{4n+2}$ & $5$ & $42$ & $143$ & $340$ & $665$ & $1150$ & $\dots$ \\ \hline
$\frac{2\binom{5n+6}{5}}{5n+2}$ & $6$ & $136$ & $728$ & $2394$ & $5980$ & $12586$ & $\dots$ \\ \hline
\end{tabular}\end{center}
The Fuss-Catalan-Raney numbers \cite{Penson} are the numbers
$$R_{p,r}(n)=\frac{r}{pn+r}\binom{pn+r}{n}.$$
We have
$$R_{p,2}(n)=\frac{2}{pn+2}\binom{pn+2}{n}.$$
Now we have
$$R_{p,2}(n)=\frac{2}{pn+2} \binom{pn+2}{n}=\frac{2}{(p-1)n+2}\binom{pn+1}{n}.$$
In general, we have the following result.
\begin{proposition} For $a_n=kn+r$, we have
$$P_n = \frac{k!}{r} R_{k+1,r}(n).$$
\end{proposition}
We can reverse this identity, to obtain a characterization of the Fuss-Catalan-Raney numbers. This says that
$R_{k,r}(n)$ is $\frac{r}{k!}$ times the gap-product sequence of the sequence $a_n=kn+r$.
\end{example}

\section{Final comments}
The formula 
$$S_n=\sum_{j=1}^{a_{n+1}-a_n-1} (a_n+j)=\frac{a_{n+1}-a_n-1}{2}\left(a_n+a_{n+1}\right)$$ may still be used when the sequence $a_n$ is not monotonic increasing, though the interpretation of the resulting sequence $S_n$ is not immediate. 
\begin{example} We look at the sequence \seqnum{A088748}, which begins 
$$1, 2, 3, 2, 3, 4, 3, 2, 3, 4, 5, 4, 3, 4, 3, 2, 3, 4, 5, 4, 5, 6, 5,\ldots.$$ 
We find that as with the natural numbers $n$, we have $S_n=0$ (this assumes the convention that a summation from $j=1$ to a negative number is empty). In the case of this sequence, the quantity $a_{n+1}-a_n-1$ is only non-zero at those indices where the paper-folding sequence \seqnum{A014707} \cite{Auto} is non-zero (and equal to $1$). At those indices, $a_{n+1}-a_n-1$ is equal to $-2$.
We now look at the sequence $\tilde{S}_n=\sum_{j=1}^{|a_{n+1}-a_n-1|} (a_n+j)$, which begins 
$$0, 0, 9, 0, 0, 11, 9, 0, 0, 0, 13, 11, 0, 11, 9, 0, 0, 0, 13, 0, 0,\ldots.$$
The sequence whose only non-zero elements of value $2 a_n-1$ occur when $a_{n+1}-a_n=-1$ begins
$$0, 0, 5, 0, 0, 7, 5, 0, 0, 0, 9, 7, 0, 7, 5, 0, 0, 0, 9, 0, 0,\ldots.$$
Subtracting these sequences, we get 
$$0, 0, 4, 0, 0, 4, 4, 0, 0, 0, 4, 4, 0, 4, 4, 0, 0, 0, 4, 0, 0,\ldots,$$ or 
$4$ times the paper-folding sequence. 

This happens since the quantity $|a_{n+1}-a_n-1|$ is equal only to $2$ when it is non-zero, the summation at indices where it is non-zero becomes $2 a_n+3$, thus giving us $2 a_n+3-(2 a_n-1)=4$ at those points where the paper-folding sequence is non-zero, and $0$ otherwise.
\end{example}

\bigskip
\hrule
\bigskip
\noindent 2020 {\it Mathematics Subject Classification}: Primary
11B83; Secondary 11B25, 11B37, 11B39.
\noindent \emph{Keywords:} Integer sequence, recurrence, generating function, gap-sum sequence, gap-product sequence, Horadam sequence, Fibonacci numbers, Jacobsthal numbers, Pell numbers, Fuss-Catalan numbers.

\bigskip
\hrule
\bigskip
\noindent (Concerned with sequences
\seqnum{A000045},
\seqnum{A000012},
\seqnum{A000040},
\seqnum{A000108},
\seqnum{A000129},
\seqnum{A001045},
\seqnum{A001710},
\seqnum{A001764},
\seqnum{A002293},
\seqnum{A002294},
\seqnum{A002295},
\seqnum{A002605},
\seqnum{A005585},
\seqnum{A014707},
\seqnum{A006002},
\seqnum{A048395},
\seqnum{A054265},
\seqnum{A088748},
\seqnum{A102693},
\seqnum{A103897} and
\seqnum{A109454}.)

\end{document}